\documentclass[psamsfonts]{amsart}

\usepackage{amssymb,amsfonts}
\usepackage{amsmath}
\usepackage[all,arc]{xy}
\usepackage{enumerate}
\usepackage{mathrsfs}
\usepackage{mathtools}
\usepackage{tikz} 
\usetikzlibrary{knots}
\usepackage{pstricks-add}
\usepackage{float}
\usepackage{graphicx}
\usepackage{subfigure}
\newtheorem{thm}{Theorem}[section]
\newtheorem{cor}[thm]{Corollary}
\newtheorem{prop}[thm]{Proposition}
\newtheorem{lem}[thm]{Lemma}

\newtheorem{quest}[thm]{Question}

\theoremstyle{definition}
\newtheorem{defn}[thm]{Definition}

\newtheorem{exmp}[thm]{Example}

\theoremstyle{remark}

\def\corr{$^{*}$\protect\footnotetext{$^{*}$ C\lowercase{orresponding Author.}}}

\makeatletter
\let\c@equation\c@thm
\makeatother
\numberwithin{equation}{section}

\title{The Lattice of subracks is atomic}

\author[A. Saki]{A. Saki}
\address{Amir Saki, Department of Pure Mathematics, Faculty of Mathematics and Computer Science, Amirkabir University of Technology (Tehran Polytechnic), 424, Hafez Ave., Tehran 15914, Iran.}
\email{amir.saki.math@gmail.com}
\author[D. Kiani]{D. Kiani\corr}
\address{Dariush Kiani, Department of Pure Mathematics, Faculty of Mathematics and Computer Science, Amirkabir University of Technology (Tehran Polytechnic), 424, Hafez Ave., Tehran 15914, Iran, and School of Mathematics, Institute for Research in Fundamental Sciences (IPM), P.O. Box 19395-5746, Tehran, Iran.}
\email{dkiani@aut.ac.ir,  dkiani7@gmail.com}
\makeatletter \newcommand*{\rom}[1]{\expandafter\@slowromancap\romannumeral #1@} \makeatother
\begin{document}
\newcommand{\tg}{\triangleright}
\newcommand{\tgl}{\triangleright^{\iota}}
\newcommand{\oo}{\overline}
\begin{abstract}
A rack is a set together with a self-distributive bijective  binary operation. In this paper, we give a positive answer to a question due to Heckenberger, Shareshian and Welker. Indeed, we prove that the lattice of subracks of a rack is atomic. Further, by using the  atoms, we associate certain quandles to racks. We also  show that the lattice of subracks of a rack is isomorphic to the lattice of subracks of a quandle. Moreover, we show that the lattice of subracks of a rack is distributive if  and only if its corresponding quandle is trivial. Finally, applying our corresponding quandles, we provide a coloring of certain knot diagrams.

\smallskip
\vspace{5pt}
\end{abstract}

\maketitle

\section{Introduction}
In 1943, a certain algebraic structure, known as \textit{key} or \textit{involutory quandle}, was introduced by M. Takasaki in \cite{Tak} to study the notion of reflection in the context of finite geometry. In 1959, J. C. Conway and G. C. Wraith introduced a more general algebraic structure  called \textit{wrack} in an unpublished correspondence. In 1982, D. ~Joyce  for the first time used the word \textit{quandle} for an algebraic and combinatorial structure to study \textit{knot invariants} \cite{Joy}. Joyce's definition of quandle is the same as the one which is nowadays used.

Let $R$ be a set together with a binary operation $\tg$ which satisfies the equality $a\tg(b\tg c)=(a\tg b)\tg(a\tg c)$, for all $a,b,c\in R$. This equality is called (\textit{left}) \textit{self-distributivity} identity. A \textit{knot} is an embedding of  $S^1$ in $\mathbb{R}^3$. In 1984, S. Matveev, and in 1986, E. Brieskorn independently used  self-distributivity systems  to study the isotopy type of braids and knots, in \cite{Mat} and \cite{Bri}, respectively. In 1992, R. Fenn and C. Rourke initiated to use the work \textit{rack} instead of wrack. They used racks to study links and knots in 3-manifolds \cite{Fen}. A rack is indeed a generalization of the concept of quandle. Racks are used to encode the movements of knots and links in the space. Knots are represented by the so-called \textit{knot diagrams}. Figure~1 is an example of a knot diagram. Homomorphisms of quandles have an interpretation as colorings of knot diagrams. 
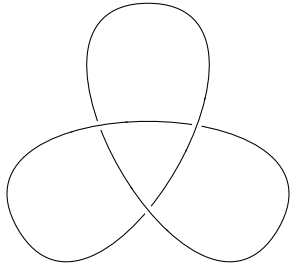
\begin{figure}\label{f1}
\begin{center}
\begin{tikzpicture}
\begin{knot}[
  consider self intersections=true,
  flip crossing=2,
  only when rendering/.style={
  }
  ]
\strand (0,2) .. controls +(2.2,0) and +(120:-2.2) .. (210:2) .. controls +(120:2.2) and +(60:2.2) .. (-30:2) .. controls +(60:-2.2) and +(-2.2,0) .. (0,2);
\end{knot}
\end{tikzpicture}
\caption{\small{Trefoil knot}}
\end{center}
\end{figure}

 Knot theory has been already applied in various areas of research, like computer science, biology, medical science and mathematical physics \cite{Mur}.

In the following, the definition of a rack and some known examples of racks are given.
\begin{defn}
A \textit{rack} $R$ is a set together with a  binary operation $\tg$ such that 
\begin{enumerate}
\item
for all $a$, $b$ and $c$ in $R$, $a\tg(b\tg c)=(a\tg b)\tg(a\tg c)$, and
\item
for all $a$ and $b$ in $R$ there exists a unique $c\in R$ with $a\tg c=b$.
\end{enumerate}
\end{defn}
Conditions (1) and (2) are called \textit{self-distributivity} and \textit{bijectivity}, respectively. A rack $R$ is called a \textit{quandle} if it satisfies the following additional condition:
 \[a\tg a=a,\quad \text{for all }a\in R.\]
It follows from the bijectivity condition of racks that the function $f_a: R\rightarrow R$ with $f_a(b)=a\tg b$ is bijective, for all $a\in R$. Therefore, by self-distributivity we have $f_a(b)\tg f_a(c)=f_af_b(c)$, for all $a,b,c\in R$.
\begin{exmp}\label{st}
The followings are some known examples of racks:
\begin{enumerate}
\item
Let $R$ be a set and $a\tg b=b$, for all $a,b\in R$. Then $R$ is a quandle,  called the \textit{trivial quandle}.
\item
Let $R$ be a set and $f$ be a permutation on $R$. Define $a\tg b=f(b)$, for all $a,b\in R$. Then $R$ is a rack, but not a quandle.
\item
Let $A$ be an abelian group and $a\tg b=2a-b$, for all $a,b\in A$. Then $A$ is a quandle, called the \textit{dihedral quandle}.
\item
Let $G$ be a group and $a\tg b=ab^{-1}a$, for all $a,b\in G$. Then $G$ is a quandle, called the \textit{core  quandle} (or \textit{rack}).
\item
Let $S=\mathbb{Z}[t,t^{-1}]$ be the ring of Laurent polynomials with integer coefficients, and $M$ be an $S$-module. Define $a\tg b=(1-t)a+tb$, for all $a,b\in M$. Then $M$ is a quandle, called the \textit{Alexander quandle}.
\item
Let $S=\mathbb{Z}[t,t^{-1},s]$ be the ring of all polynomials over $\mathbb{Z}$ with the variables $s,t,t^{-1}$ such that $t$ is invertible with the inverse $t^{-1}$. Assume that $R=S/\left<s^2-s(1-t)\right>$, and $M$ is an $R$-module. Let $x\tg y=\overline{s}x+\overline{t}y$, for all $x,y\in M$, where $\overline{s}$ and $\overline{t}$ denote $s+\left<s^2-s(1-t)\right>$ and $t+\left<s^2-s(1-t)\right>$, respectively. Then $M$ is a rack, called the \textit{$(s,t)$-rack}. It is easy to observe that an $(s,t)$-rack is not a quandle, whenever $s$ is not invertible. Note that if $s$ is invertible, then it follows from $s^2=s(1-t)$ that $s=1-t$, and hence $M$ is the Alexander quandle. One could see that $(2,-1)$-racks and dihedral racks are the same. 
\end{enumerate}
\end{exmp}

\begin{exmp}
Two knots $K$ and $K'$ are called \textit{equivalent} or \textit{ambient isotopic}, if there is a  continuous map $F:\mathbb{R}^3\times[0,1]\to\mathbb{R}^3$ such that 
\begin{enumerate}
\item
for any $t\in[0,1]$ and $x\in\mathbb{R}^3$, the map $F_t(x)=F(x,t)$ is a homeomorphism, 
\item
$F_0=id_{\mathbb{R}^3}$,
\item
$F_1(K)=K'$.
\end{enumerate}

In the above definition, $F$ is called an \textit{equivalence} or \textit{ambient isotopy}.\\
A knot diagram is indeed the projection of the image of a knot on a plane. Figure~1 is an example of a knot diagram called trefoil. 
It is known that two diagrams represent  a same knot if and only if we can obtain one of them from the other one by a finite sequence of three types of movements, called \textit{Reidemeister} moves (see Figure 2).
\begin{figure}\label{F}
\begin{center}
\subfigure{}
\includegraphics[scale=.5,trim={0 3cm 0 3.5cm},clip]{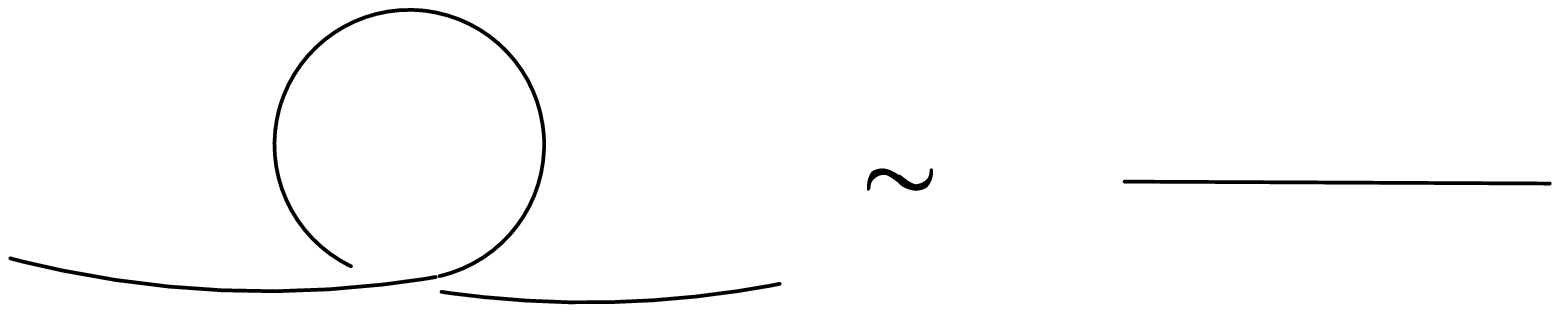}
\subfigure{}
\includegraphics[scale=0.5,trim={0 2.5cm 1.5cm 2.5cm},clip]{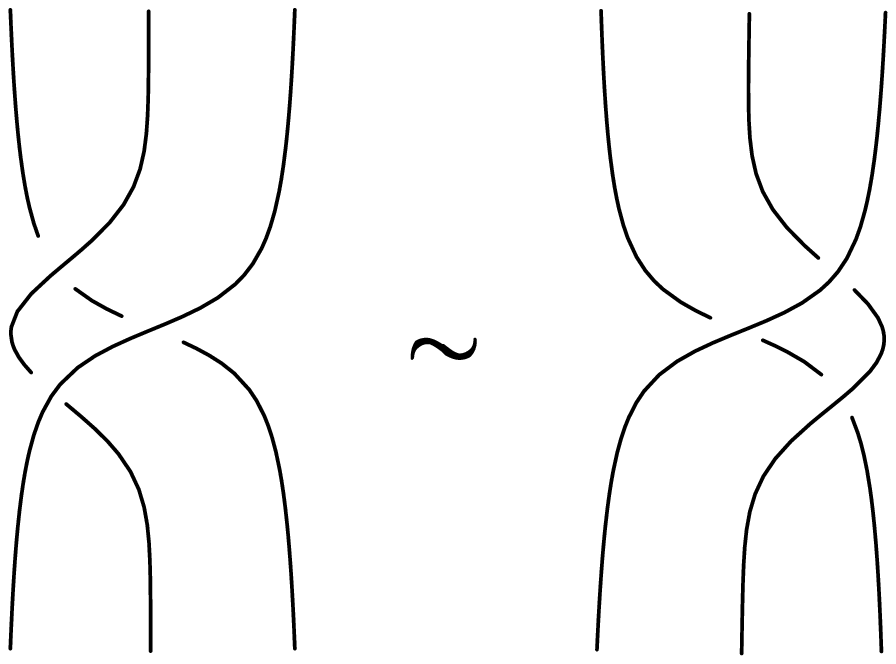}
\subfigure{}
\includegraphics[scale=.5,trim={0 1.25cm 1.1cm 1.8cm},clip]{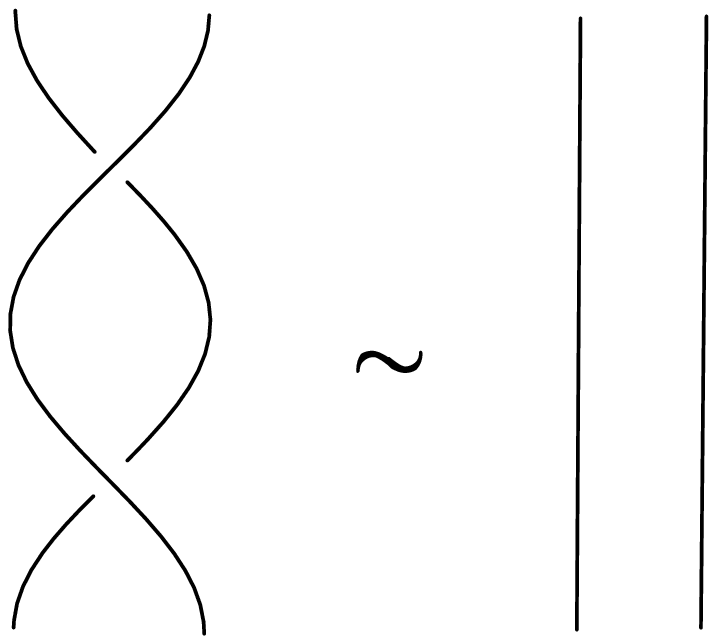}
\caption{\small{Reidemeister moves}}
\end{center}
\end{figure}

Let $K(t)$ be a parametrization of a knot $K$. Then a natural orientation is assigned to $K$ (as $t$ increases). 
Orientations of knot  diagrams determine two types of crossings: positive or negative.  Consider two strands in a diagram which  cross each other (see Figure 3). Suppose that someone stands on the  strand on top whose face is in the direction of this strand. According to Figure 3, the strand which lies under the other one could be considered as two strands: one in the left-hand side of the person, and one  in the right-hand side. If the direction of the bottom strand  is from the right-hand side to the left-hand side of the person, then we have the positive crossing. Otherwise, we have the negative crossing.
\begin{figure}\label{fig3}
\includegraphics[scale=.5,trim={0 3cm 0 3cm},clip]{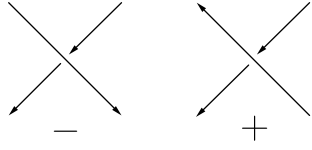}
\caption{\small{Negative and positive crossings}}
\end{figure}

In knot theory, quandles are applied to distinguish knots by \textit{coloring} knot diagrams. In fact, let $(Q,\tg)$ be a quandle. Also, assume that $\tg^{-1}$ is a binary operation on $Q$. We can assign an element of $Q$ as a color to a strand of a knot diagram.  To do this, we simply  identify the strands and their colors.

Consider a positive (resp. negative) crossing as shown in Figure 4. By assigning the color $x$ to the top strand and the color $y$ to the right (resp. left) side strand, we assign the color $x\tg y$ (resp. $x\tg^{-1}y$) to the left (resp. right) side strand. 
\begin{figure}
\begin{center}
\includegraphics[scale=.55,trim={0 3.2cm 0 3.7cm},clip]{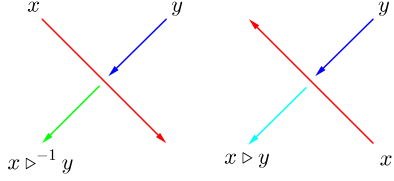}
\caption{\small{Assigning colors to strands}}
\end{center}
\end{figure}
As shown in Figure 5, by compatibility of Reidemeister moves and quandle conditions, we have $f_x^{-1}(y)=x\tg^{-1}y$, where $f_x^{-1}$ is the inverse of $f_x$ with $f_x(y)=x\tg y$, for all $x,y\in Q$.
\begin{figure}
\begin{center}
\subfigure{}
\includegraphics[scale=.5,trim={0 3cm 0 4cm},clip]{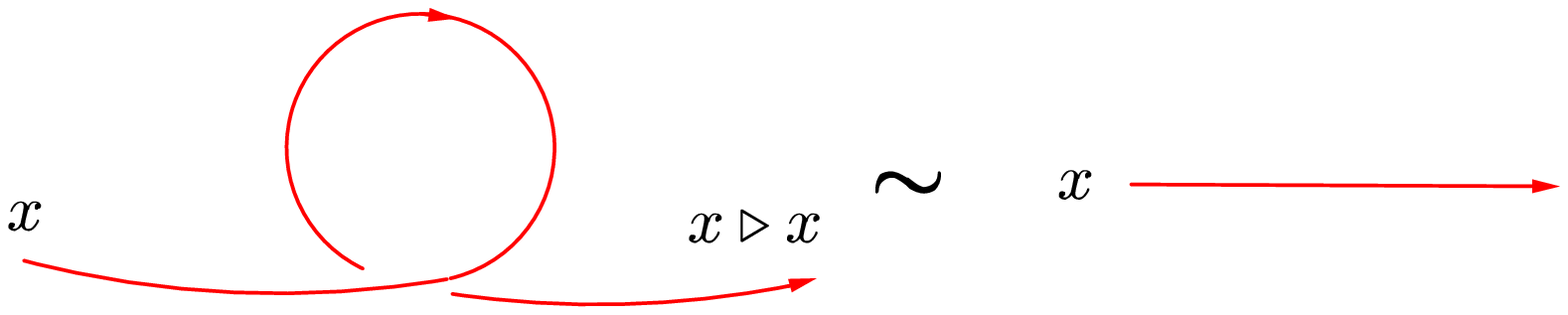}
\subfigure{}
\includegraphics[scale=.6,trim={0 1.5cm 0 3cm},clip]{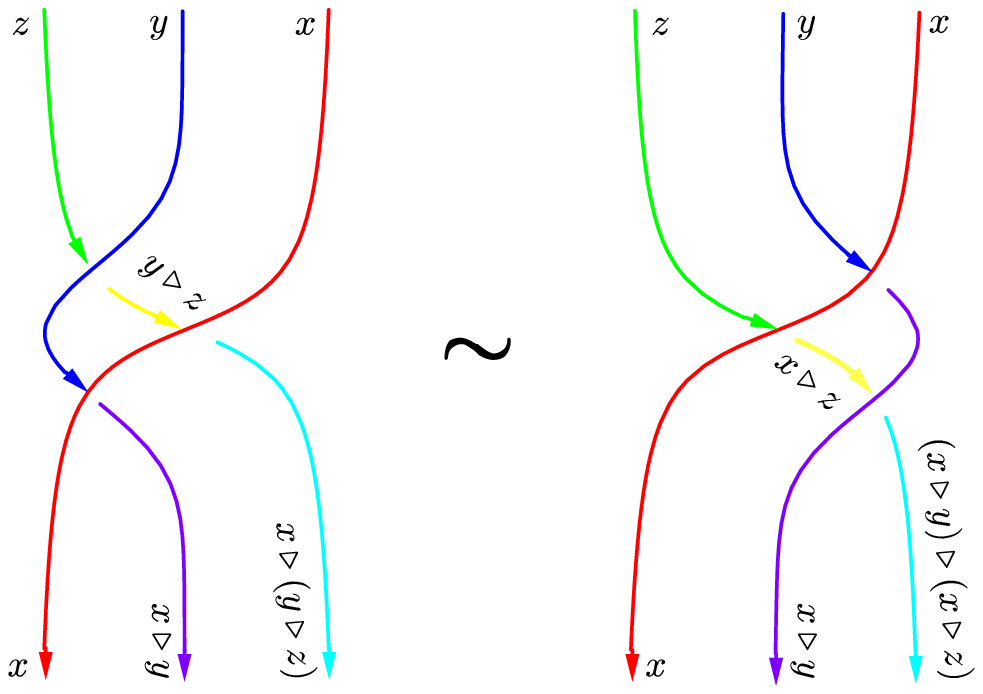}
\subfigure{}
\includegraphics[scale=.55,trim={0 1.2cm 0 3cm},clip]{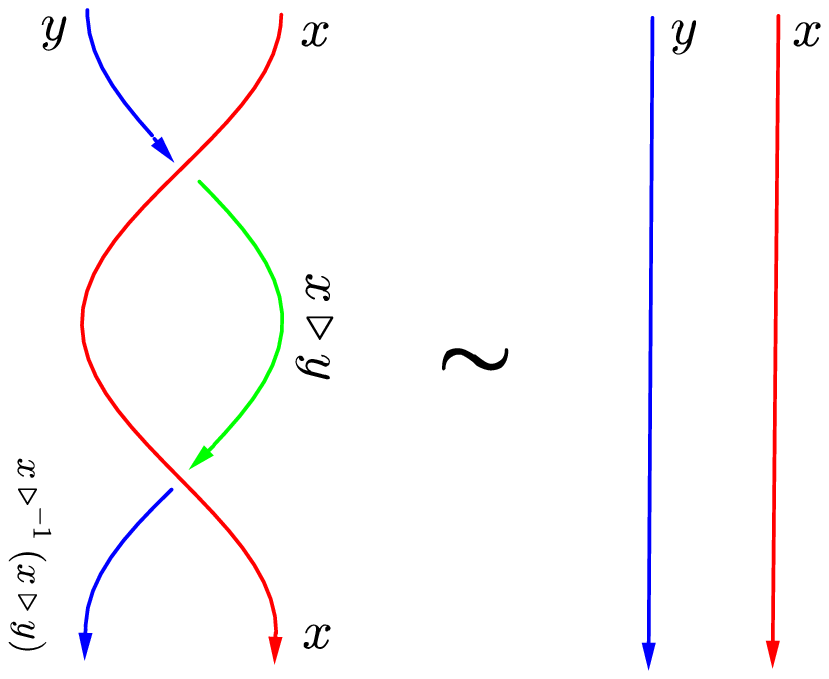}
\caption{\small{Compatibility of Reidemeister moves and quandle conditions}}
\end{center}
\end{figure}

As an example, we show that two oriented knots $5_1$ and $5_2$ in Figure~6 and Figure~7, respectively, are not equivalent.
To do this, we use dihedral quandle on $\mathbb{Z}_5$ to get the coloring given in Figure~6.
\begin{figure}\label{f8}
\begin{center}
\includegraphics[scale=.45,trim={0 -.2 0.5cm 0},clip]{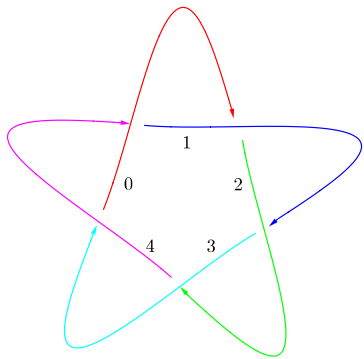}
\caption{\small{A coloring of $5_1$ by the dihedral quandle on $\mathbb{Z}_5$}}
\end{center}
\end{figure}
But we can not color $5_2$ with dihedral quandle $\mathbb{Z}_5$, such that at least two distinct colors are used. Indeed, suppose that we can color $5_2$ as  shown in Figure 7.
\begin{figure}\label{f8}
\begin{center}
\includegraphics[scale=.5,trim={0 .8cm 1.5cm 1.2cm},clip]{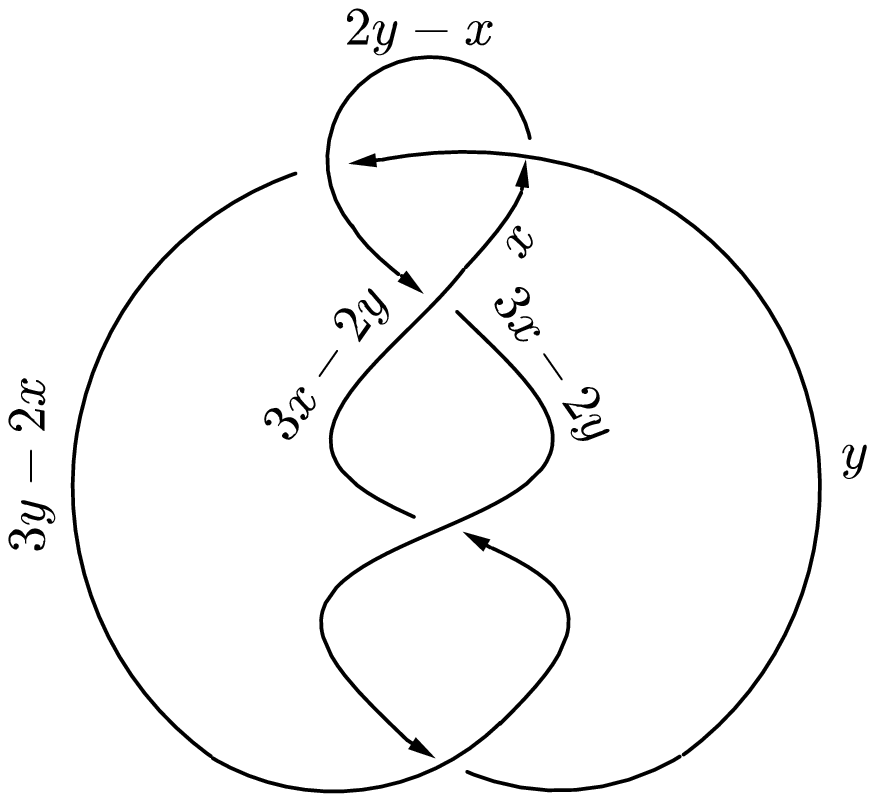}
\caption{\small{A coloring of $5_2$ by the dihedral quandle on $\mathbb{Z}_{5}$}}
\end{center}
\end{figure}
Thus $3x-2y=x$, and hence $x=y$. 
\end{exmp}

In the next example, we provide a new example of a rack which is not a quandle.

\begin{exmp}\label{ex1}
For any integers  $a$ and $b$, we define 
\[ a\tg b=\left\{\begin{array}{ll}b& ,\text{if }b \text{ is even},\\ b+2& ,\text{if }b \text{ is odd.}\end{array}\right.\]
Then it is observed that $\mathbb{Z}$ together with the above binary operation is a rack which is not a quandle.
\end{exmp}

Let $(R,\tg)$ be a rack. A subset $Q$ of $R$ is called a \textit{subrack} of $R$ if $(Q,\tg)$ is a rack. The poset of all subracks of  $R$, denoted by $\mathcal{R}(R)$, together with the inclusion relation is a lattice. I. Heckenberger et al. \cite{wel}  showed that the order complex of $\mathcal{R}(R)$ is not Cohen-Macaulay in general.

Let $G$ be a group. Define  $a\tg b=aba^{-1}$, for all $a,b\in G$. Then $(G,\tg)$ is a quandle. This rack was also studied in \cite{wel} by I. Heckenberger et al. where they considered some sublattices of $\mathcal{R}(G)$ and specified their homotopy types. For example, let $Q$ be the subrack of all transpositions of $S_n$. Then $\mathcal{R}(Q)$ is isomorphic to $\Pi_n$ which is the lattice of all partitions of a set with $n$ elements. It is known that $\Pi_n$ has the homotopy type of a wedge of $(n-2)$-spheres. As another example discussed in \cite{wel}, let $p$ be an odd prime number and $n>4$ be an integer  with $2p\le n$. Assume that $\Pi_{n,p}$ is the sublattice of all elements $B=B_1|B_2|\cdots|B_t$ of $\Pi_n$ such that $|B_i|=1$ or $|B_i|\ge p$, for all $1\le i\le t$. If $L$ is the subrack of all $p$-cycles in the alternative group $A_n$, then $\mathcal{R}(L)$ is isomorphic to $\Pi_{n,p}$, and hence it has the homotopy type of a wedge of spheres of (possibly) different dimensions.

In the last section of \cite{wel}, some questions were posed by the authors concerning the lattice of subracks of $R$. Among them, we focus on the following question:
\begin{quest}[{\cite[Question~1]{wel}}]\label{Q}
Is $\mathcal{R}(R)$ atomic for all racks $R$?
\end{quest}

We recall the definition of an \textit{atomic lattice} in the following. Let $L$ be a lattice with the least element $0$. An element $a\in L$ is called an \textit{atom} whenever $x<a$ implies that $x=0$. Then $L$ is called atomic if every element of $L$ is the join of its atoms. \\

This paper is organized as follows. In Section 2, we prove our main results. First, we prove that the lattice of subracks of any rack is atomic, which gives a positive answer to Question \ref{Q}. Next, we define a certain binary operation on the set of the atoms of a rack. Then, we show that the set of atoms together with this operation is a quandle. Moreover, we show that the lattice of subracks of this quandle is isomorphic to the lattice of subracks of the rack from which the quandle has been obtained. Furthermore, we show that the lattice of subracks of a rack is distributive if and only if its corresponding quandle is trivial. In  Section 3, we discuss a certain type of racks, called the $(s,t)$-racks. In particular, we determine their atoms. Finally, we use the obtained quandles from $(s,t)$-racks to color  certain knot diagrams.
\section{Main Results}
In this section, we prove our main results. First, we show that the lattice of subracks of a rack is atomic, which gives a positive answer to Question \ref{Q} (posed in \cite{wel}). For this purpose, the following lemmas are needed. 
\begin{lem}\label{0}
Let $R$ be a rack. For any $a$ and $b$ in $R$ we have 
\begin{enumerate}
\item
$f_{f_a(b)}=f_af_bf_a^{-1}$,
\item
$f_{f_a^{-1}(b)}=f_a^{-1}f_bf_a$.
\end{enumerate}
\end{lem}
\begin{proof}
Let $c\in R$. Then by self-distributivity 
\[f_{f_a(b)}(c)=f_a(b)\tg c=f_a(b)\tg f_af_a^{-1}(c)=f_a(b\tg f_a^{-1}(c))=f_af_b f_a^{-1}(c).\]
Thus $f_{f_a(b)}=f_af_bf_a^{-1}$ which proves (1). To prove the second equality, we have 
\[ f_a(f_a^{-1}(b)\tg (c))=(f_af_a^{-1}(b))\tg f_a(c)=b\tg f_a(c)=f_bf_a(c).\]
Therefore $f_{f_a^{-1}(b)}(c)=f_a^{-1}(b)\tg c=f_a^{-1}f_bf_a(c)$, and hence $f_{f_a^{-1}(b)}=f_a^{-1}f_bf_a$.
\end{proof}
It follows easily from Lemma \ref{0} that $f_{f_a(b)}^{-1}=f_af_b^{-1}f_a^{-1}$ and $f_{f_a^{-1}(b)}^{-1}=f_a^{-1}f_b^{-1}f_a$. 
Let $S$ be a subset of a rack $R$. The  subrack generated by $S$ in $R$, denoted by $\ll S\gg$, is defined to be the intersection of all subracks of $R$ containing $S$. For two racks $R$ and $R'$, a map $\phi:R\to R'$ is called a rack homomorphism if $\phi(a\tg b)=\phi(a)\tg\phi(b)$, for all $a,b\in R$. A bijective rack homomorphism is called a rack automorphism. For any $a\in R$, $f_a$ is an automorphism, since $f_a(b\tg c)=f_a(b)\tg f_a(c)$, for any $b,c\in R$, by self-distributivity. The set of all automorphisms of  $R$ is denoted by $\mathrm{Aut}(R)$, and is a subgroup of the group of all permutations on $R$. The subgroup generated by the set $\{f_a:\;a\in R\}$ is called the inner group of $R$ and is denoted by $\mathrm{Inn}(R)$. The inner group of $R$ acts on $R$ with the natural action $\phi*a=\phi(a)$, with $\phi\in\mathrm{Inn}(R)$. The orbits of this action are called orbits of $R$. Let $G=\mathrm{Inn}(R)$. We denote the orbit containing $a\in R$ by $Ga$. Moreover, for a subset $S\subseteq R$ and a subgroup $H$ of $\mathrm{Inn}(R)$ we set
\[ HS=\bigcup_{s\in S}Hs.\]
\begin{lem}\label{1}
Let   $R$ be a rack. Then
\begin{enumerate}
\item 
the orbits of $R$ are subracks of $R$, and
\item
if $S\subseteq R$ and $H$ is the subgroup of $\mathrm{Inn}(R)$ generated by $\{f_s:\;s\in S\}$, then $\ll S\gg =HS$.
\end{enumerate}
\end{lem}
\begin{proof}
(1) Let $a\in R$ and $G=\mathrm{Inn}(R)$. We show that $Ga$ is a subrack of $R$. Let $x=f_{a_1}^{\epsilon_1}f_{a_2}^{\epsilon_2}\cdots f_{a_t}^{\epsilon_t}(a)$ and $y=f_{b_1}^{\epsilon_1'}f_{b_2}^{\epsilon_2'}\cdots f_{b_l}^{\epsilon_l'}(a)$ be two arbitrary elements of $Ga$ for which $\epsilon_i$ and $\epsilon_j'$ are 1 or $-1$, for all $i,j$. Then by Lemma \ref{0} we have
\begin{align*}
x\tg y&=f_{f_{a_1}^{\epsilon_1}f_{a_2}^{\epsilon_2}\cdots f_{a_t}^{\epsilon_t}(a)}f_{b_1}^{\epsilon_1'}f_{b_2}^{\epsilon_2'}\cdots f_{b_l}^{\epsilon_l'}(a)
=f_{a_1}^{\epsilon_1}\cdots f_{a_t}^{\epsilon_t}f_af_{a_t}^{-\epsilon_t}\cdots f_{a_1}^{-\epsilon_1}f_{b_1}^{\epsilon_1'}\cdots f_{b_l}^{\epsilon_l'}(a),
\end{align*}
which implies that $x\tg y\in Ga$. Now, it is enough  to prove that for all $x,y\in Ga$, there exists an element $z\in Ga$ for which $f_x(z)=y$. \\
Let $x=f_{a_1}^{\epsilon_1}f_{a_2}^{\epsilon_2}\cdots f_{a_t}^{\epsilon_t}(a)$ and $y=f_{b_1}^{\epsilon_1'}f_{b_2}^{\epsilon_2'}\cdots f_{b_l}^{\epsilon_l'}(a)$. Then  
\[ z= f_{a_1}^{\epsilon_1}\cdots f_{a_t}^{\epsilon_t}f_a^{-1}f_{a_t}^{-\epsilon_t}\cdots f_{a_1}^{-\epsilon_1}f_{b_1}^{\epsilon_1'}f_{b_2}^{\epsilon_2'}\cdots f_{b_l}^{\epsilon_l'}(a)\]
is an element of $Ga$ and $x\tg z=y$. Therefore $Ga$ is a subrack of $R$.\\
(2) Let $x=f_{a_1}^{\epsilon_1}f_{a_2}^{\epsilon_2}\cdots f_{a_t}^{\epsilon_t}(s_1)$ and $y=f_{b_1}^{\epsilon_1'}f_{b_2}^{\epsilon_2'}\cdots f_{b_l}^{\epsilon_l'}(s_2)$ be two elements of $HS$. Then we have
\[x\tg y=f_{a_1}^{\epsilon_1}\cdots f_{a_t}^{\epsilon_t}f_{s_1}f_{a_t}^{-\epsilon_t}\cdots f_{a_1}^{-\epsilon_1}f_{b_1}^{\epsilon_1'}\cdots f_{b_l}^{\epsilon_l'}(s_2),\]
and hence $x\tg y\in HS$. Moreover, if 
\[ z= f_{a_1}^{\epsilon_1}\cdots f_{a_t}^{\epsilon_t}f_{s_1}^{-1}f_{a_t}^{-\epsilon_t}\cdots f_{a_1}^{-\epsilon_1}f_{b_1}^{\epsilon_1'}f_{b_2}^{\epsilon_2'}\cdots f_{b_l}^{\epsilon_l'}(s_2),\]
then $x\tg z=y$. Therefore $HS$ is a subrack of $R$, and hence $\ll S\gg\subseteq HS$. The other inclusion, follows easily from the definition of $HS$.
\end{proof}

The following theorem plays a key role in our main result.
\begin{thm}\label{2}
Let $R$ be a rack and $a\in R$. Then
\begin{enumerate}
\item
$\ll a\gg=\{f_a^n(a):\;n\in\mathbb{Z}\}$, and 
\item
if $Q$ is a subrack of $R$ such that $Q\cap \ll a\gg\neq\emptyset$, then $\ll a\gg\subseteq Q$. 
\end{enumerate}
\end{thm}

\begin{proof}
(1) It follows from Lemma \ref{1} that $\ll a\gg=Ha$ where $H$ is the subgroup of $\mathrm{Inn}(R)$ generated by $f_a$. Thus $H=\{f_a^n:\;n\in\mathbb{Z}\}$, and hence $\ll a\gg=\{f_a^n(a):\;n\in\mathbb{Z}\}$.\\
(2) Let $Q$ be a subrack of $R$ with $Q\cap \ll a\gg\neq\emptyset$, and let $f_a^{n_0}(a)\in Q\cap \ll a\gg$, for some $n_0\in\mathbb{Z}$. We show that $\ll f_a^{n_0}(a)\gg=\ll a\gg$. For $n_0=0$, there is  nothing to prove. Let $n_0\neq 0$. By Lemma \ref{1} we have $\ll f_a^{n_0}(a)\gg=H f_a^{n_0}(a)$ such that $H$ is the subgroup of $\mathrm{Inn}(R)$ generated by $ f_{f_a^{n_0}(a)}$. By Lemma \ref{0} we have 
\[ f_{f_a^{n_0}(a)}=f_{f_a^{\epsilon|n_0|}(a)}=\underbrace{f_a^{\epsilon}\cdots f_a^{\epsilon}}_{|n_0| \text{time(s)}}f_a\underbrace{f_a^{-\epsilon}\cdots f_a^{-\epsilon}}_{|n_0| \text{time(s)}}=f_a,\]
such that $\epsilon=\frac{n_0}{|n_0|}$. Thus $H$ is the cyclic subgroup of  $\mathrm{Inn}(R)$ generated by $f_a$. Consequently,
\[\ll f_a^{n_0}(a)\gg=\{f_a^{n+n_0}(a):\;n\in\mathbb{Z}\}=\{f_a^n(a):\;n\in\mathbb{Z}\}=\ll a\gg.\]
Finally, $\ll a\gg=\ll f_a^{n_0}(a)\gg\subseteq Q$,  since $f_a^{n_0}(a)\in Q$. 
\end{proof}
Let $R$ be a rack. For any $a,b\in R$, we define: $a\sim b$ if and only if $\ll a\gg=\ll b\gg$. It is clear that this is an equivalence relation on $R$. We denote the desired equivalence classes by $\oo{a}$, for all $a\in R$. It follows from Theorem \ref{2} that $\oo{a}=\ll a\gg$, for any $a\in R$. For any $A\subseteq R$, let $\oo{A}=\{\oo{a}:\;a\in A\}$. We also define the binary operation $*$ on $\oo{R}$ such that $\oo{a}*\oo{b}=\oo{a\tg b}$, for any $a,b\in R$. Using the aforementioned notation, we have the following theorem:
\begin{thm}
Let $R$ be a rack. Then $(\oo{R},*)$ is a quandle.
\end{thm}

\begin{proof}
First, note that in the proof of  Theorem \ref{2}, we  proved that for any integer $m$ and $a\in R$, we have 
\begin{equation}\label{e}
f_{f_a^n(a)}^m=f_a^m. 
\end{equation}
Now, we show that the operation $*$ is well-defined. Let $a,b\in R$,  $x\in\oo{a}$ and $y\in\oo{b}$. Thus $x=f_a^n(a)$ and $y=f_b^m(b)$ for some integers $m,n$. First, assume that $m=0$. It follows from  \ref{e} that
\[x\tg y=f_a^n(a)\tg b=f_{f_a^n(a)}(b)=a\tg b.\]
Now, assume that $m\neq 0$. By  \ref{e}, we have
\begin{align*}
x\tg y&= f_a^n(a)\tg f_b^m(b)=a\tg f_b^m(b)=f_af_b^{\epsilon|m|}(b)=f_af_b^{\epsilon}f_a^{-1}f_af_b^{\epsilon(|m|-1)}(b)\\
&=f_{f_a(b)}^{\epsilon}\left(f_af_b^{\epsilon(|m|-1)}(b)\right),
\end{align*}
where $\epsilon= \frac{m}{|m|}$. Therefore by  induction, we have $x\tg y=f_{f_a(b)}^{\epsilon |m|}(f_a(b))$, and hence $\oo{x\tg y}=\oo{a\tg b}$.
Therefore $*$ is well-defined.

Now, we show that $(\oo{R}, *)$ is a quandle. The self-distributivity condition is inherited from $(R,\tg)$. Let $c=f_a^{-1}(b)$. To show bijectivity condition, first note that we have  $\oo{a}*\oo{c}=\oo{a\tg c}=\oo{b}$.
To prove uniqueness of $\oo{c}$, let $x\in R$ with $\oo{a}*\oo{x}=\oo{b}$. Then $\oo{a\tg x}=\oo{b}$, and hence $f_a(x)=f_b^k(b)$ for some integer $k$. This implies that $x=f_a^{-1}f_b^k(b)$. Therefore, similar to the proof of well-definedness of $*$, we have the following: 
\[x=\left\{\begin{array}{ll}f_{c}^{\epsilon |k|}(c)&,\text{if }k\neq 0,\\ c&,\text{if } k=0,\end{array}\right.\]
where $\epsilon= \frac{k}{|k|}$. Therefore $\oo{x}=\oo{c}$. This completes the proof of bijectivity condition. Finally, the binary operation $*$ satisfies the quandle condition. Indeed, we have $\oo{a}*\oo{a}=\oo{a\tg a}=\oo{a}$.
\end{proof}
For a rack $R$, we refer to the quandle $\oo{R}$ as \textit{ the corresponding quandle} of $R$. 

Now, we are ready to answer Question \ref{Q} as one of our main results.
\begin{cor}
The lattice of subracks of a rack is atomic. 
\end{cor}

\begin{proof}
Let $R$ be a rack. It follows from Theorem \ref{2} that the set of atoms of $\mathcal{R}(R)$ consists of subracks $\ll a\gg$, for all $a\in R$.  Moreover, for any subrack $Q$  of $R$, we have 
\[ Q=\bigvee_{\oo{a}\in\oo{Q}}\ll a\gg.\]
\end{proof}
The following corollary shows that the lattice of subracks of $R$ and $\oo{R}$ are indeed similar. For this purpose, we use this fact that for any homomorphism $\phi:R\to S$ of racks, the image of any subrack of $R$, and the pre-image of any subrack of $S$, are subracks of $S$ and $R$, respectively. Moreover, any subrack of $S$ is the image of a subrack of $R$, whenever $\phi$ is surjective.

\begin{cor} 
Let $(R,\tg)$ be a rack and $(\oo{R},*)$ be its corresponding quandle. Then the map $Q\mapsto\oo{Q}$ defines an isomorphism from $\mathcal{R}(R)$ to $\mathcal{R}(\oo{R})$.
\end{cor}

\begin{proof}
We have the natural surjective homomorphism $\pi: R\to \oo{R}$ which sends an element $a\in R$ to $\oo{a}\in \oo{R}$. Therefore  for any subrack $Q$ of $\oo{R}$, the set $\oo{Q}$ is a subrack of $\oo{R}$. Moreover, any subrack of $\oo{R}$ is of the form of $\oo{Q}$, for some subrack $Q$ of $R$. To prove that this map is injective, assume that $\oo{Q}=\oo{Q'}$, for two subracks $Q$ and $Q'$ of $R$. For any $x\in Q$, we have $\oo{x}\in\oo{Q}$, and hence there exists an element $x'\in Q'$ with $\oo{x}=\oo{x'}$. Given that $x'\in Q'$, we conclude that $\oo{x}\subseteq Q'$, and hence $x\in Q'$. Thus we obtain $Q\subseteq Q'$. We can conclude that $Q'\subseteq Q$ in a similar way. Therefore $Q=Q'$ and the map is injective. 
\end{proof}
Note that  the above relationship between a rack $R$  and its corresponding quandle reduces the study of the lattice of subracks of $R$ to the quandle's. In \cite{Bri}, a certain quandle was associated to a rack whose lattice of subracks is not isomorphic to the one for $R$. In the following, we  discuss this correspondence. Let $(R,\tg)$ be a rack and $\iota:R\to R$ be defined by $\iota(a)=f_a^{-1}(a)$.  We show that $\iota$ is an isomorphism of racks. Let $a,b\in R$. It follows from self-distributivity condition that \[a\tg(\iota(a)\tg\iota(b))=(a\tg\iota(a))\tg(a\tg\iota(b))=a\tg(a\tg\iota(b)).\] 
Now, bijectivity condition of $(R,\tg)$ guarantees that $\iota(a)\tg\iota(b)=a\tg\iota(b)$. Moreover, we have
\[(a\tg b)\tg(a\tg\iota(b))=a\tg(b\tg\iota(b))=a\tg b.\] 
Therefore $\iota(a\tg b)=a\tg\iota(b)$, and hence $\iota(a\tg b)=a\tg\iota(b)=\iota(a)\tg\iota(b)$. To show that $\iota$ is injective, assume that $\iota(a)=\iota(b)$. Thus, we have \[a=a\tg\iota(a)=\iota(a)\tg\iota(a)=\iota(b)\tg\iota(b)=b\tg\iota(b)=b.\] 
It follows from  $a=a\tg\iota(a)=\iota(a\tg a)$   that $\iota$ is surjective, and hence $\iota\in \mathrm{Aut}(R)$. Now, we can consider $R$ together with the binary operation 
\[a\tgl b=a\tg \iota(b)=\iota(a)\tg \iota(b),\quad \text{for all }a,b\in R. \]
We show that $(R,\tgl)$ is a quandle. The quandle condition follows from $a\tgl a=a\tg\iota(a)=a$. Note that we have the following
\[a\tgl f_a^{-1}(\iota^{-1}(b))=\iota(a\tg f_a^{-1}(\iota^{-1}(b)))=b.\] 
Moreover, it follows from $a\tgl c=b$ that $\iota(a\tg c)=b$, and hence $a\tg c=\iota^{-1}(b)$. Consequently, we have $c=f_a^{-1}(\iota^{-1}(b))$. Therefore bijectivity condition is satisfied. Self-distributivity condition is obtained as follows:
\begin{align*} 
a\tgl(b\tgl c)&=a\tgl(b\tg\iota(c))=a\tg(\iota(b)\tg\iota^2(c))=(a\tg\iota(b))\tg(a\tg\iota^2(c))\\
&=(a\tgl b)\tg(\iota(a)\tg\iota^2(c))=(a\tgl b)\tgl (a\tg\iota(c))=(a\tgl b)\tgl (a\tgl c).
\end{align*}
 Using the above construction, one could easily see that any subrack of $(R,\tg)$ is  a subrack of $(R,\tgl)$ as well.  But the converse is not true. For instance, if $(R,\tg)$ is the rack defined in Example \ref{ex1}, then $(R,\tgl)$ is the trivial quandle, and hence it has some subracks, like $\{1\}$, which are not subracks of $(R,\tg)$. It follows that  $\mathcal{R}\left((R,\tg)\right)$ is a proper sublattice of the finite  lattice $\mathcal{R}\left((R,\tgl)\right)$, and hence we have $\mathcal{R}\left((R,\tg)\right)\ncong \mathcal{R}\left((R,\tgl)\right)$.\\

As an application of our results, in the following theorem, we characterize all racks $R$ for which $\mathcal{R}(R)$ is distributive. Recall that a lattice $L$ is called distributive, if the following holds:
\[ a\wedge(b\vee c)=(a\wedge b)\vee(a\wedge c),\quad\text{for all }a,b,c\in L.\]
\begin{thm}
Let $(R,\tg)$ be a rack. Then $\mathcal{R}(R)$ is distributive  if and only if $(\oo{R},*)$ is the trivial quandle.
\end{thm}

\begin{proof}
First, suppose that $\oo{R}$ is the trivial quandle. For any $a,b\in R$, we have $f_a(b)=b\in\oo{b}$, and hence $\oo{a}\cup\oo{b}$ is a subrack of $R$. Therefore  subracks of $R$ are arbitrary unions of the atoms of $R$. In particular,  the union of two subracks of $R$ is also a subrack of $R$.  This implies that the join of two subracks of $R$ is the union of them. Consequently  $\mathcal{R}(R)$ is distributive.

Conversely, assume that $\mathcal{R}(R)$ is distributive and $a,b\in R$. For any $c\in R\backslash(\oo{a}\cup\oo{b})$, we have 
\[\oo{c}\wedge\left(\oo{a}\vee\oo{b}\right)=\left(\oo{c}\wedge\oo{a}\right)\vee\left(\oo{c}\wedge\oo{b}\right)
=\left(\oo{c}\cap\oo{a}\right)\vee\left(\oo{c}\cap\oo{b}\right)=\emptyset,\]
and hence $c\notin\ll a, b\gg$. Thus $\ll a,b\gg=\oo{a}\cup\oo{b}$. It follows that $f_a(b)\in\oo{b}$, and hence $\oo{a}*\oo{b}=\oo{b}$. Therefore  $\oo{R}$ is the trivial quandle.
\end{proof}

The above theorem implies that the lattice of subracks of the rack defined in Example \ref{ex1} is distributive. For a non-distributive rack we provide a new example of racks which is a generalization of the rack defined  in Example \ref{ex1}. 
\begin{exmp}
Let $R$ be a set and $\{R_i\}_{i\in I}$ be a partition of $R$. Suppose that $\{f_i\}_{i\in I}$ is a family of bijective functions on $R$ such that
\begin{itemize}
\item
$f_i(R_j)=R_j$, and
\item
$f_if_j=f_jf_i$
\end{itemize}
for all $i,j\in I$. We define $a\tg b=f_i(b)$, for all $a\in R_i$ and $b\in R$.
Then we show that $(R,\tg)$ is a rack. To observe self-distributivity condition, let $a,b,c\in R$ with $a\in R_i$ and $b\in R_j$. So, we have 
\[ a\tg(b\tg c)=a\tg f_j(c)=f_if_j(c)=f_jf_i(c)=f_i(b)\tg f_i(c)=(a\tg b)\tg f_i(c)=(a\tg b)\tg(a\tg c).\]
To prove bijectivity condition, let $x$ be an element of $R$ for which $f_i(x)=b$. Therefore   
\[ a\tg x=f_i(x)=b.\]
To prove uniqueness of $x$, assume that $a\tg y=b$ for some $y\in R$.  Then $f_i(y)=b$ which implies $y=f_i^{-1}(b)=x$.\\
As a particular case of this structure, one can consider  $f$  to be a permutation on  $R$ and $f_i=f^i$.
\end{exmp}

\section{Coloring of knot diagrams via $(s,t)$-racks}
We have already introduced a  well known class of examples of racks in  Example~1.2, Part (6), called $(s,t)$-racks. In the following, we determine the atoms of $(s,t)$-racks as well as their corresponding quandles. To do this, we benefit from the next lemma. 

\begin{lem}\label{calc}
Let  $h\in\mathbb{Z}[t]$ and $g\in\mathbb{Z}[s]$. Suppose that $s,t$ are integers with $s^2=s(1-t)$. Then we have 
\[ h(t)g(s)=h(1-s)\left(g(s)-g_0\right)+g_0h(t)\]
where $g_0$ is the constant coefficient of $g$.
\end{lem}

\begin{proof}
It is enough to consider polynomials $h(t)$ of the form of $h(t)=t^i$, for some non-negative integer $i$.  For $i=0$, the statement is obvious. Let $i>0$. Since $s|(g(s)-g_0)$, we have
\begin{align*}
 t^ig(s)&=t^i(g(s)-g_0)+g_0t^i=t^{i-1}(ts)\left(\frac{g(s)-g_0}{s}\right)+g_0t^i\\
 &=t^{i-1}(s-s^2)\left(\frac{g(s)-g_0}{s}\right)+g_0t^i=t^{i-1}(1-s)\left(g(s)-g_0\right)+g_0t^i.
 \end{align*}
 Now, by induction, we conclude that  $t^ig(s)=(1-s)^ig(s)+g_0t^i$.
\end{proof}
Let $R=\mathbb{Z}[t,t^{-1},s]/\left<s^2-s(1-t)\right>$ as in Example \ref{st}, and let $M$ be an $(s,t)$-rack, and $a,b\in M$. We know that $f_a(b)=sa+tb$, and $\oo{a}=\{f_a^k(a)\,:\;k\in\mathbb{Z}\}$ (here, for simplicity  we  used $s$ and $t$ instead of $\oo{s}$ and $\oo{t}$, respectively). In the next lemma we determine the elements of the atoms of $M$ in terms of $s,t$.

\begin{lem}\label{f}
Let $M$ be an $(s,t)$-rack. Then
\[ f_a^k(a)=\left\{\begin{array}{ll}\left(t^k+1-(1-s)^k\right)a&,k\ge 0,\\\dfrac{(1-s)^{-k}}{t^{-k}}a&,k<0,\end{array}\right.\]
for all $a\in M$ and integers $k$.
\end{lem}

\begin{proof}
The statement is clear for $k=0$. Let $k\ge 0$. By induction and Lemma \ref{calc}, it follows that
\begin{align*}
f_a^{k+1}(a)&=f_a\left(f_a^k(a)\right)=f_a\left(\left(t^k+1-(1-s)^k\right)a\right)=sa+t\left(t^k+1-(1-s)^k\right)a\\
&=\left(t^{k+1}+s+t\left(1-(1-s)^k\right)\right)a=\left(t^{k+1}+s+(1-s)\left(1-(1-s)^k\right)\right)a\\
&=\left(t^{k+1}+1-(1-s)^{k+1}\right)a.
\end{align*}
Also, we have  $f_a\left(f_a^{-(k+1)}(a)\right)=f_a^{-k}(a)$, and hence $f_a^{-(k+1)}(a)=t^{-1}\left(f_a^{-k}(a)-sa\right)$. Therefore 
\begin{align*}
f_a^{-(k+1)}(a)&=\frac{1}{t}\left(\frac{(1-s)^k}{t^k}a-sa\right)=\frac{1}{t}\left(\frac{(1-s)^k-t^ks}{t^k}\right)a\\
&=\frac{1}{t}\left(\frac{(1-s)^k-(1-s)^ks}{t^k}\right)a=\frac{(1-s)^{k+1}}{t^{k+1}}a.
\end{align*}
\end{proof}
\begin{cor}
Let $M$ be an $(s,t)$-rack, and $a\in M$. Then 
\[ \oo{a}=\{\left(t^k+1-(1-s)^k\right)a\,:\; k\in\mathbb{Z}\text{ and } k\ge 0\}\cup \{\frac{(1-s)^k}{t^k}a\,:\; k\in\mathbb{Z}\text{ and } k\ge 0\}.\]
In particular, if $M$ is finite, then 
\[ \oo{a}=\{\left(t^k+1-(1-s)^k\right)a\,:\; k\in\mathbb{Z}\text{ and } k\ge 0\}= \{\frac{(1-s)^k}{t^k}a\,:\; k\in\mathbb{Z}\text{ and } k\ge 0\}.\]
\end{cor}

Let $R=\mathbb{Z}[t,t^{-1},s]/\left<s^2-s(1-t)\right>$  as before. The ring $\mathbb{Z}_n$ together with the scalar multiplication $f(t,t^{-1},s)\cdot a=f(t_0,t_0^{-1},s_0)a$ is  an $R$-module for which $t_0$ is  invertible with respect to the multiplication operation, and $s_0$ is an element of $\mathbb{Z}_n$ with $s_0^2=s_0(1-t_0)$. So $\mathbb{Z}_n$ is an $(s,t)$-rack. For simplicity, we use $t$ and $s$ instead of $t_0$ and $s_0$, respectively. Therefore  $\mathbb{Z}_n$ together with the binary operation $a\tg b=sa+tb$ is a rack  where $s$ and $t$ are integers with $\mathrm{gcd}(t,n)=1$. If $s$ is also invertible, then $s=1-t$, because $s^2=s(1-t)$. In this case, $\mathbb{Z}_n$ is an Alexander quandle.  Now, we are interested in cases where the corresponding quandle of $\mathbb{Z}_n$ is not an Alexander quandle.
\begin{prop}\label{prop}
Consider $\mathbb{Z}_n$ as an $(s,t)$-rack. Then in the corresponding quandle of $\mathbb{Z}_n$, we have $f_{\oo{0}}=id$   if and only if $s^2=0$  and there exists an integer $k$ such that $t^{k+1}=1-ks$.
\end{prop}

\begin{proof}
First, suppose that $f_{\oo{0}}=id$. It follows from $f_{\oo{0}}(\oo{1})=\oo{1}$ that $\oo{t}=\oo{1}$, and hence there exists a non-negative integer $k$ with $t=t^{-k}(1-s)^k$  (here, $t^{-1}$ is the inverse of $t$ in $\mathbb{Z}_n$). Since $st=s-s^2$, we have $t^{-k}s(1-s)^k=s-s^2$. It follows from Lemma \ref{calc} that $s(1-s)^k=t^ks$. Therefore $s^2=0$, and hence $(1-s)^k=1-ks$.

Conversely, let $s^2=0$ and $t^{k+1}=1-ks$  for some integer $k$. Due to the finiteness of $\mathbb{Z}_n$, we can assume that $k$ is non-negative. Thus $t^{k+1}=(1-s)^k$ which implies that $t=t^{-k}(1-s)^k$, and hence $f_{\oo{0}}(\oo{a})=\oo{a}$  for all $a\in\mathbb{Z}_n$.
\end{proof}

\begin{cor}\label{cor}
Consider $\mathbb{Z}_n$ as an $(s,t)$-rack with $s^2=0$, $s\neq 0$ and $t^{k+1}=1-ks$  for some integer $k$. Then the corresponding quandle of $\mathbb{Z}_n$ is not an Alexander quandle.
\end{cor}

\begin{proof}
Let $M$ be a finite Alexander quandle. We claim that $M$ is the trivial quandle or $f_x\neq id$ for any $x\in M$. Assume that $M$ is not the trivial quandle. If $f_x=id$  for some $x\in M$, then we have $(1-t)x=(1-t)y$  for any $y\in M$. Thus $f_x(y)=y$, for all $x,y\in M$. It follows that $M$ is the trivial quandle which is a contradiction. Therefore $f_x\neq id$ for any $x\in M$.\\
Now consider $\mathbb{Z}_n$ as an $(s,t)$-rack  with $s^2=0$, $s\neq 0$ and $t^{k+1}=1-ks$. It follows from Proposition \ref{prop} that $f_{\oo{0}}=id$. We show that $f_{\oo{1}}\neq id$. Note that if  $f_{\oo{1}}= id$, then we have $f_{\oo{1}}(\oo{0})= \oo{0}$. It follows that $\oo{s}=\oo{0}$ while $\oo{0}=\{0\}$, and hence $s=0$ which is a contradiction. Therefore, the  corresponding quandle of  $\mathbb{Z}_n$ is not an Alexander quandle.
\end{proof}
As an example which satisfies the assumptions of Corollary \ref{cor}, let $n=p_1^{\alpha_1}\cdots p_r^{\alpha_r}$ such that $\alpha_i>1$  for some $1\le i\le r$. We can apply $s=p_1^{\beta_1}\cdots p_r^{\beta_r}$ with $\beta_j=\lfloor\frac{\alpha_j+1}{2}\rfloor$  for $1\le j\le r$. We  have $s\neq 0$ and $s^2=0$. By applying $t=1$  and $k=s$, it follows from   Corollary \ref{cor} that the corresponding quandle of $\mathbb{Z}_n$ as an $(s,t)$-rack is not an Alexander quandle.

\begin{exmp}
Let $s=3$ and $t=1$, and define $a\tg b=3a+b$ on $\mathbb{Z}_9$. Then the atoms of $\mathcal{R}(\mathbb{Z}_9)$ are the following:
\[ \oo{0}=\{0\},\quad \oo{1}=\{1,4,7\},\quad \oo{2}=\{2,5,8\},\quad \oo{3}=\{3\},\quad \oo{6}=\{6\}.\]
Now, the corresponding quandle of this rack is given by $f_{\oo{0}}=f_{\oo{3}}=f_{\oo{6}}=id$ and  $ f_{\oo{1}}= f_{\oo{2}}^{-1}=(\oo{0}\;\,\oo{3}\;\,\oo{6})\in S_5$.
\end{exmp}
\begin{figure}
\begin{center}
\includegraphics[scale=.27,trim={0 2cm 0 2cm},clip]{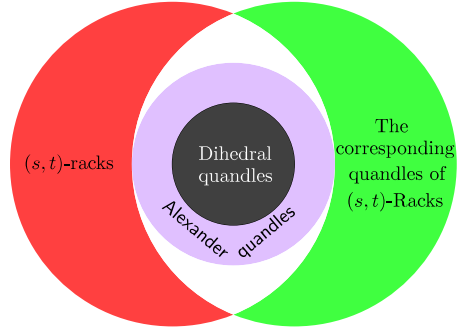}
\caption{\small{$(s,t)$-racks versus their corresponding  quandles}}
\end{center}
\end{figure}

Despite the corresponding quandles  of  $(s,t)$-racks are not necessarily Alexander quandles, there are some similarities between them and Alexander quandles, as we see in the following proposition:
\begin{prop}
Let $M$ be an $(s,t)$-rack. Then in the  corresponding quandle of $M$, we have $\oo{x}*\oo{y}=\oo{sx+(1-s)y}$, for all $x,y\in M$.
\end{prop}

\begin{proof}
By Lemma \ref{calc}, we have 
\[ \frac{(1-s)^{\phi(n)-1}}{t^{\phi(n)-1}}\left(sx+(1-s)y\right)=\frac{s(1-s)^{\phi(n)-1}}{t^{\phi(n)-1}}x+\frac{1}{t^{\phi(n)-1}}y=sx+ty,\]
where $\phi$ is the Euler's phi function.
\end{proof}
Finally, as an application, we color certain knot diagrams using the corresponding quandles of  $(s,t)$-racks. 
Let $s=2$ and $t=9$, and define $a\tg b=2a+9b$ on $\mathbb{Z}_{20}$. We can get the coloring given in Figure 9 for the oriented $5_1$. But one could not color the oriented $5_2$ by the corresponding quandle of $\mathbb{Z}_{20}$ in which at least two distinct colors are used. Indeed, suppose that we could do so, as shown in Figure 10. Thus $\oo{8x-7y}=\oo{x}$, and hence $8x-7y=x$ or $8x-7y=11x$, because $\oo{x}=\{x,11x\}$. Therefore $\oo{x}=\oo{y}$ which is a contradiction. This observation shows that the two oriented knots $5_1$ and $5_2$ are not equivalent.
\begin{figure}[H]
\begin{center}
\includegraphics[scale=.45,trim={0 -.4cm 0.5cm 0},clip]{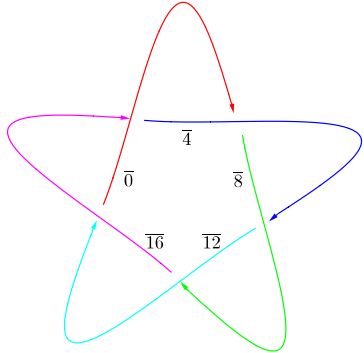}
\caption{\small{A coloring of oriented $5_1$ by the corresponding quandle  of $\mathbb{Z}_{20}$ as a $(2,9)$-rack}}
\end{center}
\end{figure}
\begin{figure}[H]
\begin{center}
\includegraphics[scale=.5,trim={0 .4cm 1.5cm 1.2cm},clip]{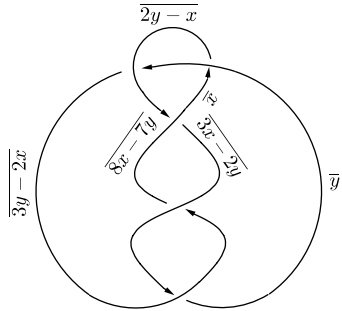}
\caption{\small{A  coloring of oriented $5_2$ by the corresponding quandle  of $\mathbb{Z}_{20}$ as a $(2,9)$-rack}}
\end{center}
\end{figure}

\section{Acknowledgements}
The authors would like to thank Sara Saeedi Madani and Volkmar Welker for their useful comments. The authors would also like to thank the institute for research in fundamental science (IPM). The research of the second author was in part supported by a grant from IPM (No. 96050212).

\end{document}